\documentclass[11pt]{article}
\usepackage[a4paper,margin=2.5cm]{geometry}
\usepackage[T1]{fontenc}
\usepackage[utf8]{inputenc}
\usepackage{lmodern}
\usepackage{amsmath,amssymb,amsthm,mathtools}
\usepackage{enumitem}
\usepackage{microtype}
\usepackage[hidelinks]{hyperref}

\newtheoremstyle{plainbreak}
  {6pt}{6pt}{\itshape}{}\bfseries{.}{0.5em}{}
\theoremstyle{plainbreak}
\newtheorem{lemma}{Lemma}
\newtheorem{proposition}{Proposition}
\newtheorem{theorem}{Theorem}
\theoremstyle{definition}
\newtheorem{definition}{Definition}
\newtheorem{example}{Example}
\newtheorem{counterexample}{Counterexample}

\title{\textbf{Formal Matrix Rings with the Involution Property}}
\author{Valeriya Kirova\\
CR4AI, University MEPHI\\
\texttt{lera.orlanovna@gmail.com}}
\date{}

\begin{document}
\maketitle

\begin{abstract}
We study formal matrix rings with the involution property, that is, rings in which every element can be represented as the sum of an invertible element and an involution. After recalling the construction of formal matrix rings via Morita contexts, we describe involutions in triangular formal matrix rings and establish a criterion for upper triangular rings of order two. We then prove that a formal matrix ring of arbitrary finite order has the involution property whenever all of its diagonal rings have this property. Several examples and counterexamples are included to illustrate the obtained results.
\end{abstract}


\vspace{0.5em}

\noindent\textit{This paper is an English translation and extended version of the work originally published in Russian in 2019 in the proceedings of the All-Russian Youth Scientific Conference “All Facets of Mathematics and Mechanics,” held in Tomsk on April 23–27, 2019, pp. 6–17.}

\medskip
\section{Introduction}

Formal matrices, also called generalized matrices, play an important role in algebra. They occur naturally in ring and module theory and provide a common framework for ordinary matrix rings, triangular matrix rings, and Morita context rings. The underlying construction goes back to Morita's work on duality for modules \cite{Morita1958}. A systematic treatment of formal matrix rings, their modules, and their structural properties was later developed by Krylov and Tuganbaev \cite{KrylovTuganbaev2009,KrylovTuganbaev2017}; see also the monograph \cite{KrylovTuganbaevSpringer2017}. More generally, ring-theoretic properties of Morita contexts and generalized matrix rings have been investigated in a number of settings, including semiperfectness, regularity, and cleanness \cite{TangLiZhou2014}.

The decomposition of ring elements as sums of special elements is another classical theme. Henriksen studied rings generated additively by their units and, in particular, decompositions of matrices into sums of units \cite{Henriksen1974}. Nicholson's theory of clean rings, where every element is a sum of a unit and an idempotent, placed such decompositions in a broader structural setting \cite{Nicholson1977}. The property considered in the present paper is stronger in a different direction: every element is required to be a sum of a unit and an \emph{involution}, that is, a unit $v$ satisfying $v^2=1$.

The involution property in this sense was studied explicitly by Goldsmith, Meehan, and Wallutis for endomorphism rings of torsion-free groups and modules \cite{GoldsmithMeehanWallutis2007}. They also showed that if a ring $R$ has the involution property, then full matrix rings over $R$ inherit it. The present note considers the analogous question for formal matrix rings, where diagonal entries may belong to different rings and off-diagonal entries to prescribed bimodules. We first recall the construction of formal matrix rings and characterize involutions in the upper triangular case. We then prove that a $2\times2$ upper triangular formal matrix ring has the involution property exactly when its diagonal rings do, and finally establish a sufficient condition for a formal matrix ring of arbitrary order: if all diagonal rings possess the involution property, then so does the entire formal matrix ring.



A Morita context is a tuple
\[
(R,M,N,S,\varphi,\psi),
\]
where $R,S$ are rings; $M,N$ are bimodules; and $\varphi$ and $\psi$ are bimodule homomorphisms satisfying certain compatibility conditions.

From a given Morita context one constructs a ring of matrices of the form
\[
K=\begin{pmatrix}R&M\\N&S\end{pmatrix}.
\]
It is called the \emph{Morita context ring}. It is also called a formal matrix ring of order two, or a generalized matrix ring.

The entries of matrices in a Morita context ring, which below we shall simply call a formal matrix ring, may belong to several rings and bimodules. Formal matrices are added and multiplied according to the usual rules of matrix addition and multiplication. These algebraic operations produce a ring: a formal matrix ring. Formal matrix rings extend the notion of the ring of $n\times n$ matrices over a given ring.

\begin{definition}
Let $R$ be a ring. A left $R$-module is an additive abelian group $A$ together with a map
\[
R\times A\to A
\]
such that for all $a,b\in A$ and $r,s\in R$,
\begin{enumerate}[label=\arabic*.]
\item $r(a+b)=ra+rb$,
\item $(r+s)a=ra+sa$,
\item $(rs)a=r(sa)$,
\item $1a=a$.
\end{enumerate}
The definition of a right $R$-module is analogous.
\end{definition}

\begin{definition}
Let $S$ be another ring. An $R$--$S$-bimodule ${}_RA_S$ is an abelian group $A$ that is simultaneously a left $R$-module ${}_RA$ and a right $S$-module $A_S$, and such that for all $a\in A$, $r\in R$, $s\in S$,
\[
(ra)s=r(as).
\]
\end{definition}

\section*{Construction of Formal Matrix Rings}
All rings will be assumed associative and unital, and all modules and bimodules unital.

We define a formal matrix ring. Let $R$ and $S$ be two rings. Let $M$ be an $R$--$S$-bimodule and $N$ an $S$--$R$-bimodule. Denote by $K$ the set of all matrices of the form
\[
K=\begin{pmatrix}r&m\\n&s\end{pmatrix},
\qquad r\in R,\ s\in S,\ m\in M,\ n\in N.
\]
With respect to matrix addition, $K$ is clearly an abelian group. However, $K$ is not yet a ring. In order to turn $K$ into a ring, we must be able to compute ``products'' of elements $n\in N$ and $m\in M$ so that the resulting products lie in $R$ and $S$:
\[
mn\in R,\qquad nm\in S.
\]
This can be done correctly as follows. Suppose bimodule homomorphisms are given:
\[
\varphi:M\otimes_S N\to R,
\qquad
\psi:N\otimes_R M\to S.
\]
For all $n\in N$ and $m\in M$, set
\[
\varphi(m\otimes n)=mn,
\qquad
\psi(n\otimes m)=nm.
\]
Then matrices in $K$ can be multiplied as in an ordinary matrix ring:
\[
\begin{pmatrix}r&m\\n&s\end{pmatrix}
\begin{pmatrix}r_1&m_1\\n_1&s_1\end{pmatrix}
=
\begin{pmatrix}
rr_1+mn_1 & rm_1+ms_1\\
nr_1+sn_1 & nm_1+ss_1
\end{pmatrix},
\]
where $r,r_1\in R$, $s,s_1\in S$, $m,m_1\in M$, and $n,n_1\in N$.

Here $rm_1$, $ms_1$, $nr_1$, and $sn_1$ are the corresponding module products. We assume that for all $m,m_1\in M$ and $n,n_1\in N$ the associativity relations
\[
(mn)m_1=m(nm_1),
\qquad
(nm)n_1=n(mn_1)
\]
hold.

With the above addition and multiplication, $K$ is a ring. In checking the ring axioms, one must also use the basic properties of the tensor product and the bimodule nature of $\varphi$ and $\psi$.

The converse is also true: if $K$ is a ring, then the stated associativity relations hold. The ring $K$ is called a \emph{formal matrix ring} and is also denoted by
\[
K=\begin{pmatrix}R&M\\N&S\end{pmatrix}.
\]
If $N=0$, then $K$ is the ring of upper triangular matrices:
\[
K=\begin{pmatrix}R&M\\0&S\end{pmatrix}.
\]
If $M=0$, then $K$ is the ring of lower triangular matrices:
\[
K=\begin{pmatrix}R&0\\N&S\end{pmatrix}.
\]
To define rings of upper or lower triangular matrices, the homomorphisms $\varphi$ and $\psi$ are not needed.

When studying a formal matrix ring $K$, it is necessary to understand how the properties of this ring depend on the properties of the rings $R$ and $S$, the bimodules $M$ and $N$, and the homomorphisms $\varphi$ and $\psi$.

It is sometimes convenient to identify matrices with their corresponding elements. For example, an element $r\in R$ can be identified with the matrix
\[
\begin{pmatrix}r&0\\0&0\end{pmatrix}.
\]
A formal matrix ring of arbitrary order $n$ can be defined analogously.

\begin{definition}
Let $R$ be a ring. An element $v\in R$ is called \emph{invertible} if there exists an element $x\in R$ such that
\[
vx=1=xv.
\]
Such an $x$ is called the inverse of $v$ and is denoted by $v^{-1}$. Thus
\[
vv^{-1}=1=v^{-1}v.
\]
The element $v^{-1}$ is unique for $v$.
\end{definition}

\begin{definition}
An element $v$ of a ring $R$ is called an \emph{involution} if
\[
v^2=1,
\]
i.e. $v^{-1}=v$. Equivalently, $v$ is an invertible element of order $2$.
\end{definition}

The simplest examples of involutions are $1$ and $-1$. Indeed,
\[
(-1)^2=1.
\]

We say that a ring $R$ has the \emph{involution property} if every element $x\in R$ can be represented in the form
\[
x=u+v,
\]
where $u$ and $v$ are invertible elements and, moreover, $v$ is an involution.

The following question arises: which formal matrix rings have the involution property? To answer it, we first determine which matrices are involutions.

\begin{lemma}
Let $K$ be the ring of upper triangular matrices
\[
K=\begin{pmatrix}R&M\\0&S\end{pmatrix},
\]
where $R$ and $S$ have the involution property. A matrix
\[
\begin{pmatrix}r&m\\0&s\end{pmatrix}
\]
is an involution if and only if $r$ and $s$ are involutions and
\[
rm=-ms.
\]
\end{lemma}

\begin{proof}
\emph{Necessity.} Suppose
\[
\begin{pmatrix}r&m\\0&s\end{pmatrix}
\]
is an involution. This means that
\[
\begin{pmatrix}r&m\\0&s\end{pmatrix}^{2}
=
\begin{pmatrix}1&0\\0&1\end{pmatrix}.
\]
Compute
\[
\begin{pmatrix}r&m\\0&s\end{pmatrix}
\begin{pmatrix}r&m\\0&s\end{pmatrix}
=
\begin{pmatrix}r^2&rm+ms\\0&s^2\end{pmatrix}.
\]
Since the original matrix is an involution,
\[
\begin{pmatrix}r^2&rm+ms\\0&s^2\end{pmatrix}
=
\begin{pmatrix}1&0\\0&1\end{pmatrix}.
\]
Hence
\[
r^2=1,\qquad s^2=1,\qquad rm+ms=0.
\]
Thus $r$ and $s$ are involutions and $rm=-ms$.

\emph{Sufficiency.} Suppose $r$ and $s$ are involutions and $rm=-ms$. Then
\[
\begin{pmatrix}r&m\\0&s\end{pmatrix}
\begin{pmatrix}r&m\\0&s\end{pmatrix}
=
\begin{pmatrix}r^2&rm+ms\\0&s^2\end{pmatrix}.
\]
Using $r^2=1$, $s^2=1$, and $rm=-ms$, we obtain
\[
\begin{pmatrix}r^2&rm+ms\\0&s^2\end{pmatrix}
=
\begin{pmatrix}1&0\\0&1\end{pmatrix}.
\]
Therefore the matrix is an involution.
\end{proof}

\begin{proposition}
Let
\[
K=\begin{pmatrix}R&M\\0&S\end{pmatrix}
\]
be the ring of upper triangular matrices of order two. Then $K$ has the involution property if and only if $R$ and $S$ have the involution property.
\end{proposition}

\begin{proof}
\emph{Necessity.} Suppose that $K$ has the involution property. Take any $r\in R$. It can be written as
\[
\begin{pmatrix}r&0\\0&0\end{pmatrix}
=
\begin{pmatrix}u_1&m_1\\0&u_2\end{pmatrix}
+
\begin{pmatrix}v_1&m_2\\0&v_2\end{pmatrix},
\]
where the first matrix is invertible and the second is an involution. Hence $u_1$ is invertible and $v_1$ is an involution (by the preceding lemma). Furthermore,
\[
r=u_1+v_1.
\]
Thus $R$ has the involution property. The proof for $S$ is analogous.

\emph{Sufficiency.} Suppose that $R$ and $S$ have the involution property. Take an arbitrary matrix in the formal matrix ring $K$:
\[
\begin{pmatrix}r&m\\0&s\end{pmatrix}\in K.
\]
By assumption, the elements $r$ and $s$ can be represented as
\[
r=u_1+v_1,
\qquad
s=u_2+v_2,
\]
where $u_1,u_2$ are invertible and $v_1,v_2$ are involutions.
Then
\[
\begin{pmatrix}r&m\\0&s\end{pmatrix}
=
\begin{pmatrix}u_1+v_1&m\\0&u_2+v_2\end{pmatrix}
=
\begin{pmatrix}u_1&m\\0&u_2\end{pmatrix}
+
\begin{pmatrix}v_1&0\\0&v_2\end{pmatrix},
\]
where
\[
\begin{pmatrix}u_1&m\\0&u_2\end{pmatrix}
\]
is invertible, while
\[
\begin{pmatrix}v_1&0\\0&v_2\end{pmatrix}
\]
is an involution. Therefore $K$ has the involution property.
\end{proof}

\begin{theorem}
Let $K$ be a formal matrix ring of order $n$,
\[
K=
\begin{pmatrix}
R_1 & M_{1,2} & \cdots & M_{1,n}\\
M_{2,1} & R_2 & \cdots & M_{2,n}\\
\vdots & \vdots & \ddots & \vdots\\
M_{n,1} & M_{n,2} & \cdots & R_n
\end{pmatrix},
\]
where $R_i$ $(i=1,\ldots,n)$ are rings. If the rings $R_i$ have the involution property, then $K$ also has the involution property.
\end{theorem}

\begin{proof}
Take an arbitrary matrix $M$ in the formal matrix ring $K$:
\[
M=
\begin{pmatrix}
r_1 & m_{1,2} & \cdots & m_{1,n}\\
m_{2,1} & r_2 & \cdots & m_{2,n}\\
\vdots & \vdots & \ddots & \vdots\\
m_{n,1} & m_{n,2} & \cdots & r_n
\end{pmatrix}.
\]
Since, by hypothesis, all rings $R_i$ have the involution property, the elements $r_i\in R_i$ $(i=1,\ldots,n)$ can be represented as
\[
r_i=u_i+v_i,
\]
where $u_i$ is invertible and $v_i$ is an involution in $R_i$.

Hence the original matrix $M$ can be written as
\[
\begin{aligned}
M
&=
\begin{pmatrix}
u_1+v_1 & m_{1,2} & \cdots & m_{1,n}\\
m_{2,1} & u_2+v_2 & \cdots & m_{2,n}\\
\vdots & \vdots & \ddots & \vdots\\
m_{n,1} & m_{n,2} & \cdots & u_n+v_n
\end{pmatrix}\\[1ex]
&=
\begin{pmatrix}
u_1 & m_{1,2} & \cdots & m_{1,n}\\
m_{2,1} & u_2 & \cdots & m_{2,n}\\
\vdots & \vdots & \ddots & \vdots\\
m_{n,1} & m_{n,2} & \cdots & u_n
\end{pmatrix}
+
\begin{pmatrix}
v_1 & 0 & \cdots & 0\\
0 & v_2 & \cdots & 0\\
\vdots & \vdots & \ddots & \vdots\\
0 & 0 & \cdots & v_n
\end{pmatrix}.
\end{aligned}
\]

We prove by mathematical induction that the formal matrix ring $K$ has the involution property.

Let $n=1$. In this case, the matrix $M$ consists of a single element $r_1$, which has the involution property. Hence it is clear that the formal matrix ring $K$ also has this property.

Now assume $n\ge 2$. Write the original matrix $M$ in block form:
\[
M=\begin{pmatrix}A&B\\C&d\end{pmatrix}.
\]
Here $A$ is an $(n-1)\times(n-1)$ matrix,
\[
A=
\begin{pmatrix}
R_1 & M_{1,2} & \cdots & M_{1,n-1}\\
M_{2,1} & R_2 & \cdots & M_{2,n-1}\\
\vdots & \vdots & \ddots & \vdots\\
M_{n-1,1} & M_{n-1,2} & \cdots & R_{n-1}
\end{pmatrix},
\]
$C$ is a row vector of size $1\times(n-1)$,
\[
C=(M_{n,1}\; M_{n,2}\; \cdots\; M_{n,n-1}),
\]
$B$ is a column vector of size $(n-1)\times1$,
\[
B=\begin{pmatrix}M_{1,n}\\M_{2,n}\\\vdots\\M_{n-1,n}\end{pmatrix},
\]
and $d$ is an element with the involution property, $d\in R_n$.

By the induction hypothesis, the matrix $A$ can be represented as the sum of two matrices,
\[
A=D+U,
\]
where $U$ is invertible and $D$ is a diagonal involution. Put
\[
a=CU^{-1}B,\qquad a\in R_n.
\]
Next write
\[
d-a=u+v,
\]
where $u$ is invertible and $v$ is an involution in $R_n$. Then $d$ can be represented as
\[
d=u+a+v.
\]
Therefore
\[
M=
\begin{pmatrix}U&B\\C&a+u\end{pmatrix}
+
\begin{pmatrix}D&0\\0&v\end{pmatrix}.
\]
Clearly, the second matrix in the sum is a diagonal involution. We show that the first matrix is invertible. Introduce two invertible matrices $P$ and $Q$:
\[
P=
\begin{pmatrix}
I&0\\
-CU^{-1}&1
\end{pmatrix},
\qquad
Q=
\begin{pmatrix}
U^{-1}&-U^{-1}Bu^{-1}\\
0&u^{-1}
\end{pmatrix}.
\]
Here $I$ is the identity matrix of order $n-1$. Denote the first matrix in the sum by
\[
V=\begin{pmatrix}U&B\\C&a+u\end{pmatrix}.
\]
Computing the product $PVQ$, we obtain
\[
PVQ=
\begin{pmatrix}I&0\\0&1\end{pmatrix}.
\]
Thus the product is the identity matrix, and therefore $V$ is invertible.

Hence $M$ is the sum of an invertible matrix and an involution. Consequently, the formal matrix ring $K$ has the involution property.
\end{proof}

\subsection{Examples and counterexamples}

We begin with several elementary examples illustrating the involution property.

\begin{example}
Let \(F\) be a field of characteristic different from \(2\). Then \(F\) has the involution property.

Indeed, the elements \(1\) and \(-1\) are involutions. For any \(a\in F\), if \(a\neq 1\), then
\[
a=(a-1)+1,
\]
and \(a-1\) is invertible. If \(a=1\), then
\[
1=2+(-1),
\]
where \(2\) is invertible because \(\operatorname{char}F\neq 2\). Thus every element of \(F\) is a sum of an invertible element and an involution.
\end{example}

\begin{example}
Consider the ring of upper triangular matrices
\[
K=
\begin{pmatrix}
\mathbb{F}_3 & \mathbb{F}_3\\
0 & \mathbb{F}_3
\end{pmatrix}.
\]
Since \(\mathbb{F}_3\) has the involution property, Proposition~6 implies that \(K\) also has the involution property.

For example,
\[
A=
\begin{pmatrix}
0&1\\
0&2
\end{pmatrix}
\]
admits the decomposition
\[
A=
\begin{pmatrix}
1&1\\
0&1
\end{pmatrix}
+
\begin{pmatrix}
2&0\\
0&1
\end{pmatrix}.
\]
The first matrix is invertible, while the second matrix is an involution, since
\[
\begin{pmatrix}
2&0\\
0&1
\end{pmatrix}^{2}
=
\begin{pmatrix}
1&0\\
0&1
\end{pmatrix}
\]
over \(\mathbb{F}_3\).
\end{example}

The assumptions above cannot be omitted.

\begin{counterexample}
The field \(\mathbb{F}_2\) does not have the involution property.

Indeed, its only invertible element is \(1\), and its only involution is also \(1\). Hence the only element representable as the sum of an invertible element and an involution is
\[
1+1=0.
\]
Therefore \(1\in\mathbb{F}_2\) has no such representation.
\end{counterexample}

\begin{counterexample}
The ring of integers \(\mathbb{Z}\) does not have the involution property.

The only invertible elements of \(\mathbb{Z}\) are \(1\) and \(-1\), and these are also its only involutions. Consequently, the possible sums of an invertible element and an involution are only
\[
-2,\qquad 0,\qquad 2.
\]
In particular, the element \(1\) cannot be represented in the required form.
Therefore \(\mathbb{Z}\) does not have the involution property.

Consequently, by Proposition~6, the triangular matrix ring
\[
\begin{pmatrix}
\mathbb{Z}&\mathbb{Z}\\
0&\mathbb{Z}
\end{pmatrix}
\]
does not have the involution property either.
\end{counterexample}


\begin{thebibliography}{9}

\bibitem{Morita1958}
K. Morita, ``Duality for modules and its applications to the theory of rings with minimum condition,'' \emph{Science Reports of the Tokyo Kyoiku Daigaku, Section A}, vol.~6, pp.~83--142, 1958.

\bibitem{Henriksen1974}
M. Henriksen, ``Two classes of rings generated by their units,'' \emph{Journal of Algebra}, vol.~31, no.~1, pp.~182--193, 1974. doi:10.1016/0021-8693(74)90013-1.

\bibitem{Nicholson1977}
W. K. Nicholson, ``Lifting idempotents and exchange rings,'' \emph{Transactions of the American Mathematical Society}, vol.~229, pp.~269--278, 1977. doi:10.1090/S0002-9947-1977-0439876-2.

\bibitem{GoldsmithMeehanWallutis2007}
B. Goldsmith, C. Meehan, and S. L. Wallutis, ``Torsion-free groups and modules with the involution property,'' \emph{Journal of Pure and Applied Algebra}, vol.~208, no.~1, pp.~127--134, 2007. doi:10.1016/j.jpaa.2005.11.006.

\bibitem{KrylovTuganbaev2009}
P. A. Krylov and A. A. Tuganbaev, ``Modules over formal matrix rings,'' \emph{Fundamental and Applied Mathematics}, vol.~15, no.~8, pp.~145--211, 2009 [in Russian]; English transl., \emph{Journal of Mathematical Sciences}, vol.~171, no.~2, pp.~248--295, 2010. doi:10.1007/s10958-010-0133-5.

\bibitem{KrylovTuganbaev2014}
P. A. Krylov and A. A. Tuganbaev, ``Formal matrices and their determinants,'' \emph{Fundamental and Applied Mathematics}, vol.~19, no.~1, pp.~65--119, 2014 [in Russian].

\bibitem{TangLiZhou2014}
G. Tang, C. Li, and Y. Zhou, ``Study of Morita contexts,'' \emph{Communications in Algebra}, vol.~42, no.~4, pp.~1668--1681, 2014. doi:10.1080/00927872.2012.748327.

\bibitem{KrylovTuganbaev2017}
P. A. Krylov and A. A. Tuganbaev, \emph{Formal Matrix Rings and Modules over Them}. Moscow: MCCME, 2017 [in Russian].

\bibitem{KrylovTuganbaevSpringer2017}
P. Krylov and A. Tuganbaev, \emph{Formal Matrices}, Algebra and Applications, vol.~23. Cham: Springer, 2017. doi:10.1007/978-3-319-53907-2.

\end{thebibliography}
\end{document}